\documentclass[reqno]{amsart}
\usepackage{amssymb}

\input xy
\xyoption{all}

\usepackage{graphics}

\usepackage{mathrsfs}

\theoremstyle{plain}

\newtheorem{lemma}{Lemma}[section]
\newtheorem{theorem}[lemma]{Theorem}

\newtheorem{corollary}[lemma]{Corollary}
\newtheorem{claim}{Claim}

\newtheorem*{stat}{\name}
\newcommand{\name}{testing}

\theoremstyle{definition}
\newtheorem{definition}[lemma]{Definition}

\newtheorem{problem}{Problem}

\theoremstyle{remark}

\newtheorem*{remark*}{Remark}

\newcommand{\qedc}{{\qed}~{\rm Claim~{\theclaim}.}}
\newcommand{\qedsc}{{\qed}~{\rm Claim.}}

\newenvironment{cproof}
{\begin{proof}[Proof of Claim.]}
{\qedc\renewcommand{\qed}{}\end{proof}}

\numberwithin{equation}{section}

\newcommand{\pup}[1]{\textup{(}{#1}\textup{)}}

\newcommand{\set}[1]{\{#1\}}
\newcommand{\setm}[2]{\set{#1\mid#2}}

\newcommand{\famm}[2]{(#1\mid#2)}

\newcommand{\gen}[1]{\langle{#1}\rangle}

\newcommand{\card}[1]{|{#1}|}

\newcommand{\sC}{\mathscr{C}}
\newcommand{\sL}{\mathscr{L}}

\newcommand{\bA}{\boldsymbol{A}}
\newcommand{\bB}{\boldsymbol{B}}
\newcommand{\bC}{\boldsymbol{C}}
\newcommand{\bD}{\boldsymbol{D}}
\newcommand{\bE}{\boldsymbol{E}}
\newcommand{\bF}{\boldsymbol{F}}
\newcommand{\bG}{\boldsymbol{G}}

\newcommand{\bM}{\boldsymbol{M}}
\newcommand{\bP}{\boldsymbol{P}}
\newcommand{\bQ}{\boldsymbol{Q}}
\newcommand{\bR}{\boldsymbol{R}}
\newcommand{\bS}{\boldsymbol{S}}

\newcommand{\cC}{\mathcal{C}}

\newcommand{\cF}{\mathcal{F}}

\newcommand{\cH}{\mathcal{H}}

\newcommand{\cL}{\mathcal{L}}

\newcommand{\cP}{\mathcal{P}}

\newcommand{\cV}{\mathcal{V}}
\newcommand{\cW}{\mathcal{W}}

\DeclareMathOperator{\Var}{{\bf{Var}}}
\DeclareMathOperator{\QVar}{{\bf{QVar}}}
\DeclareMathOperator{\Hom}{Hom}
\DeclareMathOperator{\dom}{dom}
\DeclareMathOperator{\Sub}{Sub}
\newcommand{\res}{\mathbin{\restriction}}

\newcommand{\undertilde}{\underset{\sim}}

\newcommand{\ignorer}[1]{}


\begin{document}

\date{\today}

\title{Finite Abelian algebras are fully dualizable}

\author[W.~Bentz]{Wolfram Bentz}
\author[P.~Gillibert]{Pierre Gillibert}
\author[L.~Sequeira]{Lu\'is Sequeira}

\address[W. Bentz]{Center of Algebra/Center for Computational and Stochastic Mathematics, Universidade de Lisboa, Lisbon, Portugal}
\email[W. Bentz]{wfbentz@fc.ul.pt}
\address[P. Gillibert]{Pontificia Universidad  Cat\'olica de Valpara\'iso, Instituto de Matem\'aticas, Valpara\'iso, Chile}
\email[P. Gillibert]{pierre.gillibert@ucv.cl, pgillibert@yahoo.fr}
\address[L. Sequeira]{Departamento de Matem\'atica, Faculdade de Ci\^encias da Universidade de Lisboa, Lisbon, Portugal}
\email[L. Sequeira]{lfsequeira@fc.ul.pt}

\begin{abstract}
We show that every  finite Abelian algebra $\bA$ from congruence-permutable varieties admits a full duality. In the process, we prove that $\bA$ also allows a strong duality, and that the duality may be induced by a dualizing structure $\undertilde{A}$ of finite type. We  give an explicit bound on the arities of the partial and total operations appearing in $\undertilde{A}$. In addition, we show that the enriched partial hom-clone of $\bA$ is finitely generated
as a clone.
\end{abstract}

\maketitle

\section{Introduction}
A full duality represents elements of abstract algebraic structures by using functions on a topological space that is often enriched with a relational and/or operational structure, and vice versa. This representation allows us to solve algebraic questions by the way of the additional structure.  For example in Stone duality, Boolean algebras are dual to Boolean spaces.
Under this correspondence, the familiar Cantor space is dual to the denumerable free Boolean algebra, with many of the universal properties of the Cantor space being dual counterparts to the natural universal properties of being a free algebra (the universal mapping property for example).

In a natural full duality, the
representation is constructed in a certain systematic way, using a generating algebra $\bA$ and
a corresponding topological structure $\undertilde{\bA}$, called an ``alter ego" of $\bA$. We say that $\bA$  is fully dualizable, if there exists an alter ego $\undertilde{\bA}$
such that every algebra
from the quasi-variety generated by $\bA$ and every topological structure from
the topological quasivariety generated by $\undertilde{\bA}$ has a representation.  We remark that
in case of a full duality, the correspondence can be extended to homomorphisms and continuous structure
preserving maps, yielding a category-theoretic dual equivalence between the corresponding categories.

A full duality is the symmetrized concept of a duality. The definitions of duality and dualizability differ from that of full duality and full dualizability
by requiring that only the algebras in the quasivariety generated by $\bA$ have duals,
while the topological quasivariety generated by $\undertilde{\bA}$ might contain structures without a representation.

Despite a growing understanding of duality theory, dualizability and full dualizability of an algebra continue
 to be mysterious properties. For some classes of algebras (such as algebras generating
congruence-distributive varieties) there exists a well-behaved dividing line between the dualizable and non-dualizable algebras. In other cases, the partial results available seem to defy any discernable pattern. This latter case includes classes of algebras that are otherwise considered to be well-understood, such as Malcev algebras (or even extensions of groups).

Abelian algebras are among the most well-behaved classes of algebras, being polynomially equivalent to modules. Surprisingly, they have until very recently resisted any attempts to obtain results concerning their dualizability. At the 4th Novi Sad Algebra Conference in 2013, Kearnes and Szendrei
presented a proof that showed the dualizability of all finite modules \cite{KS13}. In an unpublished result,
Bentz and Peter Mayr extended their argument
to finite modules with all constants, which is equivalent to showing dualizability for all finite
abelian algebras with all constants. Finally, Gillibert proved the dualizability of
all finite Abelian algebras \cite{Gil14}, answering a question from \cite{BM13}. The same result was independently shown by Kearnes and Szendrei \cite{KSpre}.

In this article, we will complete the remaining dualizability question for abelian algebras by showing the following Theorem.
\begin{theorem}\label{T:main}Let $\bA$ be an Abelian algebra generating a congruence-modular variety. Then $\bA$ is fully dualizable.
\end{theorem}
 In fact we will show slightly more. Firstly,
we show that a full duality can be obtained by an alter ego $\undertilde{\bA}$
 of finite type, and we give an explicit bound
on the arity of the (partial) functions and relations in $\undertilde{\bA}$. Secondly, we establish
full dualizability by showing that every finite Abelian algebra satisfies the stronger property of
(adequately named) strong dualizability.

Additionally, we obtain a structural result in clone theory by showing that the clone of all partial functions compatible with an Abelian algebra $\bA$ is finitely generated as a clone (Corollary~\ref{C:ClonePartielFinimentEngendre}).

The proof of our main theorem relies on a technical condition from \cite{LMW} (Theorem \ref{T:enoughTAO}), that requires us to find a suitable factorization for each partial $\bA$-compatible function through a bounded set of
partial functions.
Our article is structured around this requirement as follows: In Section \ref{S:basics} we define basic terms and establish several results about Abelian algebras. Section \ref{S:factoring} provides a technical result about the factorization
of projections on partial domains in the quasivariety generated by $\bA$. This result will allow us to concentrate our further considerations on partial homomorphisms without proper extensions.
In Section \ref{S:Extensions}, we prove a crucial theorem about those partial homomorphisms: namely, a partial homomorphism that cannot be extended must have a large domain. This
result is then used in Section \ref{S:factorPartial} to prove a factorization property  for all partial homomorphisms, and to prove our main theorem.

Section \ref{S:Examples} and Section \ref{S:Problems} contain an example calculation and a list of problems motivated by our research.
Moreover, we have included an appendix that gives explicit bounds on the number of various algebraic objects. While the results of the appendix are used in our arguments, they are only necessary in establishing an explicit bound
on the arities used in a fully dualizing alter ego. A reader without an interest in such an explicit bound may ignore the appendix and instead  check the simple fact that all quantities in our argument are finite.

For simplicity,  we will not provide definitions of the various types of dualities in
this article, relying on established technical results to prove our claims.
For definitions and a general background on duality theory, we refer to the standard work
\cite{CD98}.

\section{Basic concepts}\label{S:basics}

Given an algebra $\bA$, we denote by $A$ its underlying set, and by  $\Sub(\bA)$ the set of subalgebras of $\bA$. The variety $\Var(\bA)$ (respectively, the quasivariety $\QVar(\bA)$)
generated by $\bA$ is the smallest classes of algebras, with the signature of  $\bA$, that is closed under taking products, subalgebras, and homomorphic images (respectively, products, subalgebras,
and isomorphic algebras).

Given $X \subseteq A$ we denote by $\gen{X}$ the subalgebra of $\bA$ generated by $X$. For an arbitrary set $X$
and a variety $\cV$, we denote by $F_\cV(X)$ the algebra freely generated by $X$ in $\cV$.

A subproduct algebra $\bA \le \Pi_i \bA_i$ is called a subdirect product if $\pi_i(A) =A_i$ for each projection $\pi_i$.
An algebra is \emph{subdirectly irreducible} if whenever it is isomorphic to a subdirect product, it is already isomorphic to one of its factors.

An algebra $\bA$ is \emph{affine} if there exists an Abelian group structure $\langle A; +, 0, -\rangle$ such that $t(x,y,z)=x-y+z$ is both
a term function of $\bA$ and a homomorphism from $\bA^3$ to $\bA$. A class of algebras $\cC$ is \emph{affine} if all of its algebras are. In the case of an affine variety $\cV$, it is easy to see that we may choose one term $t$ that witnesses the affinity simultaneously for all members of $\cV$ (e.g. we could take the term witnessing the affinity of $F_{\cV}(\omega)$).

An algebra $\bA$ is \emph{Abelian} if $[1_A,1_A]=0_A$, where
$1_A$ and $0_A$ are the universal and trivial relations on $A$, and $[\cdot, \cdot]$ denotes the binary commutator on the
congruences of $\bA$ (we refer to \cite{FM} for the definition of the commutator). A class of algebras is Abelian if all of its members are.
As usual, when dealing with commutator theoretic
conditions, we restrict to algebras that generate congruence-modular varieties. With this condition, Abelian algebras and varieties coincide
with affine algebras and varieties \cite[Corollary 5.9]{FM}, and we will use the two notations interchangeably throughout the paper. Our results will rely exclusively on the defining property of affine algebras.

We repeat several results about congruences of  Abelian algebras from \cite{Gil14}.

\begin{definition}[\cite{Gil14}, Definition 3.1] Let $\bA$ be an Abelian algebra and $\bB \in \Sub(\bA)$. The \textbf{\emph{congruence generated}} by $B$, denoted by $\Theta_B$ is the smallest congruence of $\bA$ containing $B^2$.
\end{definition}

We remark that not every congruence of $\bA$ can be written in the form $\Theta_{\bB}$ for some subalgebra $\bB$, and that we might have $\Theta_{\bB} = \Theta_{\bC}$ with $\bB \ne \bC$.

\begin{lemma}[\cite{Gil14}, Lemma 3.3]\label{L:ThetaDiscr}
Let $\bA$ be an Abelian algebra, let $\bB \in \Sub(\bA)$, and let $t$ be a term witnessing the affinity of $\bA$. Then
$$\Theta_B=\setm{(x,y) \in A^2}{\exists b \in B, t(x,y,b) \in B}\,.$$
\end{lemma}

Note that this result implies that $B$ is a congruence class of $\Theta_{\bB}$.

\begin{lemma}[\cite{Gil14}, Corollary 3.7]\label{L:MeetIrred}
Let $\bA$ be an Abelian algebra and let $\bB \in \Sub(\bA)$ such that $B$ is meet irreducible in the semilattice $\langle \Sub(\bA); \cap\rangle$. Then $\bA / \Theta_B$ is subdirectly irreducible.\end{lemma}

We will next establish several results about varieties of Abelian algebras and their relationship with varieties of modules. We remark that our result can be obtained from more general results in commutator theory, see in particular \cite[Chapter 9]{FM}. As the bound estimates in our results require explicit descriptions of the constructions involved in the proofs of these auxiliary lemmas, we have decided to provide self-contained arguments.

It is well known that any Abelian variety is polynomially equivalent to a variety of modules. The following lemma gives a corresponding result, stating that a variety of
Abelian algebras with at least one nullary symbol is term-equivalent to a variety of modules extended with constants. Here and throughout, we will always consider all rings to be rings with unit.

\begin{lemma}\label{L:SeparerConstantes}
Let $\cV$ be a variety of Abelian algebras \pup{over a language $\sL$}. Assume that $c\in\sL$ is a nullary  operation symbol. Then there exists a language $\sL'\cup\sC$
 and an interpretation of $\sL'\cup\sC$ on the underlying sets of algebras in $\cV$ with the properties listed below. Here, for $\bA \in \cV$, $\bA'$ denotes the algebra obtained from $\bA$ by considering the interpretation of $\sL' \cup \sC$, and $\cV'$ will denote the collection of these algebras.
\begin{enumerate}
\item $\sC$ contains only nullary symbols.
\item For every operation symbol $f \in \sL'\cup\sC$ there exists a term $s_f$ over $\sL$, such that
 $f^{\bA'}=s_f^{\bA}$ for all $\bA \in \cV$.
\item For every  $ f\in \sL$ there exists a term $s_f$  over $\sL'\cup\sC$, such that  $f^{\bA}=s_f^{\bA'}$ for all $\bA \in \cV$.
\item $\sL'$ may be considered as a module language over a suitable ring $\bR$.
\item The reduct of $\cV'$ to $\sL'$ is a variety of $\bR$-modules.
\item For every $\bA \in \cV$, $c^{\bA}$ denotes the neutral element of the module $\bA'$.
\end{enumerate}
\end{lemma}

\begin{proof}
Set
\begin{equation*}
\sC=\setm{d}{d\text{ constant symbol in }\sL}\cup\setm{f(c,\dots,c)}{f\text{ an operation symbol in }\sL}\,,
\end{equation*}
and set  $d^{\cV'}= d^{\cV}$ for all $d \in \sC$.

Let $+,-,0$, be binary, unary, and nullary operation symbols, respectively. We set
$$
x+^{\cV'}y=t^\cV(x,c^\cV,y)\,,
$$
$$
-^{\cV'}y=t^\cV(c^\cV,y,c^\cV)\,,
$$
and $0^{\cV'}=c^\cV$.

Given an operation symbol {$f\in\sL$} of arity $n\ge 1$, we consider the $\sL$-terms $f_1,\dots,f_n$ given by
\begin{equation}\label{E:ringterms}
f_i(x)=f(c,\dots,c,x,c,\dots,c)-f(c,\dots,c)\,,\quad\text{}
\end{equation}
where the $x$ appears at position $i$.
Let $R$ be the closure under $+$, $-$, and $\circ$ of all the $f_i$ and the term $x$ modulo the equational theory of $\cV$.
As $t$ acts as a homomorphism of $\cV$, it is easy to check that the induced
actions of  $+$, $-$, $0$, and $\circ$ on $R$ are well-defined and give to $R$ the structure of a ring with unit $[x]$. Moreover, each element
of $R$ is a set of $\cV$-equivalent unary terms in $\sL$. Hence we may interpret any element of $R$ on
$\cV'$ in the same way as any of its member terms interprets in  $\cV$.

Set $\sL'=R\cup\set{+,-,0}$. Note that $\sL'$ can be considered as an $R$-module language,
and it is easy to check that the interpretations of $\sL'$ induce an $R$-module structure on $\cV$.
 By construction, each operation in $\sL'\cup\sC$ interprets as some term in $\sL$.

Conversely, all constant operations of $\sL$ are in $\sC$, and so interpret as $\sL' \cup \sC$-terms.
Moreover, if $f$ is an $n$-ary operation symbol in $\sL$, then
$$
f^{\cV}(x_1,\dots,x_n)=[f_1]^{\cV'}(x_1)+^{\cV'}[f_2]^{\cV'}(x_2)+{\cV'}\dots+^{\cV'}[f_n]^{\cV'}(x_n)+^{\cV'}(f(c,c,\dots,c))^{\cV'}\,,
$$
which is the interpretation of a term in $\sL'\cup\sC$.
The last assertion is obvious.
\end{proof}

The following lemma will be used to expand varieties by adding a constant operation. The structure of the resulting variety is similar to the original variety. Combining this lemma with Lemma~\ref{L:SeparerConstantes} shows us that varieties of Abelian algebras are almost varieties of modules.

\begin{lemma}\label{L:AjouterConstantes}
Let $\sL$ be a similarity type. Pick a constant operation symbol $c$ which is not in $\sL$. Given an $\sL$-algebra $\bA$ and $x\in A$, we consider $\bA_x$ to be the expansion of $\bA$ to the language $\sL\cup\set{c}$, where
the new constant element is interpreted as $x$, that is $c^{\bA_x}=x$. The following statements hold.
\begin{enumerate}
\item If $f\colon \bA\to \bB$ is a morphism of $\sL$-algebras and $x\in A$, then $f\colon \bA_x\to \bB_{f(x)}$ is a morphism of $\sL\cup\set c$-algebras.
\item If $\cV$ is variety of $\sL$-algebras, then $\cV_c=\setm{\bA_x}{\bA\in\cV\text{ and }x\in A}$ is a variety of $\sL\cup\set c$-algebras.
\item If $\cV$ is affine, then $\cV_c$ is also affine.
\item If $\bF$ is freely generated by $x_1,\dots,x_n,y$ over $\cV$, then $\bF_y$ is freely generated by $x_1,\dots,x_n$ over $\cV_c$.
\item If $\cV$ is locally finite, then $\cV_c$ is also locally finite.
\end{enumerate}
\end{lemma}

\begin{proof}
Let $\cV$ be a variety of $\sL$-algebras.
\begin{enumerate}
\item Clearly, $f$ preserves the structure of $\sL$-algebras. Moreover $f(c^{\bA_x})=f(x)=c^{\bB_{f(x)}}$, therefore $f$ is a morphism of $\sL\cup\set c$-algebras.
\item Denote by $\cW$ the variety of $\sL\cup\set c$-algebras satisfying the equational theory of $\cV$. Then $\cV_c\subseteq\cW$.

Conversely, let $\bA\in\cW$, denote by $\bA(\sL)$ the reduct of $\bA$ to $\sL$. Note that $\bA=\bA(\sL)_{c^{\bA}}$. Therefore $\cV_c=\cW$ is a variety of algebras.
\item Let $\cV$ be affine, as witnessed by the $\sL$-term $t$. Then $t$ is a term in $\sL\cup\set{c}$ that
witnesses the affinity of $\cV_c$.
\item First note that $\bF_y$ is generated by $x_1,\dots,x_n$. Pick an algebra in $\cV_c$. It can be written as $\bA_v$ where $\bA\in\cV$ and $v\in\bA$. Let $u_1,\dots,u_n \in \bA$. As $x_1,\dots,x_n,y$ freely generate $\bA$, there is a morphism $f\colon\bF\to\bA$ such that $f(x_i)=u_i$ for all $1\le i\le n$, and $f(y)=v$. Hence $f\colon\bF_y\to\bA_v$ is a morphism.
\item Follows immediately from (4).\qed
\end{enumerate}
\renewcommand{\qed}{}
\end{proof}

\begin{corollary}\label{C:ringsizeAbelian}
Suppose  that $\cV$ is generated by an Abelian algebra $\bA$ of type $\cL$, where $\card A=p_1^{\alpha_1}\dots p_k^{\alpha_k}$ for
distinct primes $p_i$. Let $\cV_c$ be the  variety constructed in Lemma \ref{L:AjouterConstantes} and $\bR$ be the ring constructed
in Lemma~\textup{\ref{L:SeparerConstantes}} with regard to the variety $\cV_c$.  Then  $\bR$ satisfies $\card R= p_1^{r_1}\dots p_k^{r_k}$, where $r_i \le \alpha_i^2$ for $i=1, \dots,k$.
\end{corollary}
\begin{proof}
All elements of $\bR$ are represented by unary terms $\lambda$ in $\cL \cup\{c\}$ that satisfy $\lambda(c) \approx c$ in $\cV_c$, as this identity holds for the generating set (\ref{E:ringterms}).
We call any such term $\lambda$ a \emph{ring term}. For each ring term $\lambda$ with variable $x$, let $\bar \lambda$ be the binary $\cL$-term over $\{x,y\}$ obtained by replacing every occurrence of $c$ with $y$.

By our proof of Lemma \ref{L:AjouterConstantes} (2), the equational theory $\bar \lambda$ is generated by the equational theory of $\cV$. It follows that any $\cV_c$-valid identity involving $c$ is obtained from
a  $\cV$-valid identity by replacing a variable with $c$. Hence, as $\cV_c \models\lambda(c) \approx c$ , $\cV \models\bar \lambda(y,y) \approx y$, and so all terms of the form $\bar \lambda$ are idempotent in $\cV$.

Let  $\lambda,\mu$ be ring terms of $\cV_c$, $a \in A$, and assume that $\bA_a \models\lambda(x) \approx \mu(x)$. We claim that $\cV_c \models \lambda(x) \approx \mu(x)$.

As $\cV_c$ is generated by $\famm{\bA_y}{y\in A}$, it suffices to show that for every $b \in A$, $\bA_b \models \lambda(x) \approx \mu(x)$.

Set $\bar \nu(x,y) = \bar \lambda(x,y) - \bar \mu(x,y) + x$, and note that this is an $\cL$-term. Let $\nu(x) = \bar \lambda(x,c) -  \bar\mu(x,c) + c$ be the corresponding $\cL \cup\{c\}$-term. Then as
$\bA_a \models\lambda(x) \approx \mu(x)$, it follows that
$\bA_a \models \nu(x) \approx c$. Hence $\bar\nu^{\bA}(x,a)=a$ for all $x$. Moreover, as $\bar \lambda$ and $\bar \mu$ are idempotent $\bar\nu^{\bA}(a,a)=a= \bar\nu^{\bA}(x,a)$ for all $x \in A$.

As $\bA$ is Abelian, we have that $\bar\nu^{\bA}(a,b)= \bar\nu^{\bA}(x,b)$ for all $x \in A$. Hence $\nu^{\bA_b}(a)= \nu^{\bA_b}(x)$ for all $x \in A$, and so $\nu$ is constant on $\bA_b$. As
$\nu^{\bA_b}(b)=b$, we have that $\bA_b \models \nu(x) \approx c$. This implies that $\bA_b \models \lambda(x) \approx \nu(x)$. As this holds for all $b \in A$,
we have $\cV_c \models \lambda(x) \approx \mu(x)$, as required.

It follows that if $\nu$ and $\mu$ represent distinct elements of $\bR$, then $\lambda^{\bA_a} \ne \nu^{\bA_a}$. Conversely, any term function $\lambda^{\bA_a}$ with  $\lambda^{\bA_a}(a)=a$ determines an element of
$\bR$. Let $\bar R$ be the set of these term functions, so that $\card{\bar R} = \card{\bR}$.

 As term functions of $\bA_a$, the elements of $\bar R$ are compatible with the functions $t^{\bA_a}$ and by construction, they are compatible with $c^{\bA_a}$. Hence every
$\lambda^{\bA_a} \in \bar R$ is an endomorphisms of the algebra $\langle A;\, t^{\bA_a}, c^{\bA_a}\rangle$. As this algebra has an underlying Abelian group structure, the result follows from Lemma \ref{L:Majorations}.
\end{proof}

Given sets $A,B$ we denote by $\cF(A,B)$ the set of all maps $A\to B$. Let $A$ be a set. Given $n\in\mathbb{N}$ we consider the set of $n$-ary partial operations defined by:
\begin{equation*}
\cC^n(A) = \bigcup\setm{\cF(X,A)}{X\subseteq A^n\,,X\not=\emptyset}\,.
\end{equation*}
The set of all partial operations over $A$ is
\begin{equation*}
\cC(A)=\set{\emptyset}\cup\bigsqcup_{n\in\mathbb{N}} \cC^n(A)\,.
\end{equation*}

Note that alternative definitions distinguish empty functions of different arity; the difference is immaterial for our results.
Denote by $\pi_i^n\colon A^n\to A$, $\vec x\mapsto x_i$ the canonical projection for all positive integers $n$ and all $1\le i\le n$. A \emph{partial clone} over a $A$ is a set $\cF\subseteq\cC(A)$, containing all projections and closed under composition of partial functions.

Let $\cF$ be a partial clone over $A$. A \emph{domain of arity $n$ of $\cF$} is a subset $D$ of $A^n$ such that there exists $f\in\cF$ such that $\dom f=D$.

\begin{lemma}\label{L:ClonePartielDomaine}
Let $\cF$ be a partial clone over a set $A$. Let $n$ be a positive integer. Let $C,D$ be domains of arity $n$ of $\cF$. The following statements hold.
\begin{enumerate}
\item For all $1\le i\le n$, $\pi_i^n\res D\in\cF$.
\item The set $C\cap D$ is a domain in $\cF$.
\item Let $p\in\cF$ of arity $n$. If $D\subseteq \dom p$, then $p\res D\in\cF$.
\item If $p=(p_1,\dots,p_n)\colon A^k\to A^n$ are in $\cF$, then $p^{-1}(D)$ is a domain of $\cF$.
\end{enumerate}
\end{lemma}

\begin{proof}
Take $f\colon C\to A$ and $g\colon D\to A$ in $\cF$.
\begin{enumerate}
\item is a special case of (\ref{I:domainrestrict}), shown below.

\item $\pi_1^2(f(\vec x),g(\vec x ))$ is defined if an only if $\vec x\in C\cap D$, thus $C\cap D$ is a domain in $\cF$.
\item $\pi_1^2(p(\vec x),g(\vec x))$ is defined if and only if $\vec x\in \dom p\cap\dom g=D$. Moreover $p(\vec x)=\pi_1^2(p(\vec x),g(\vec x))$, for all $\vec x\in D$, therefore $p\res D\in\cV$. \label{I:domainrestrict}
\item
The domain of $g\circ p$ is $p^{-1}(D)$, therefore $p^{-1}(D)$ is a domain of $\cF$.\qed
\end{enumerate}
\renewcommand{\qed}{}
\end{proof}

We will consider partial functions whose domains are subalgebras and which are homomorphisms.
For algebras $\bA \le \bB$ and $\bC$ and a homomorphism $f$ from $\bA$ to~$\bC$, we say that $f$ has a proper extension if
there is an algebra $\bD$ with $\bA < \bD \le \bB$ and a homomorphism $f'$ from $\bD $ to $\bB$ that extends $f$.

As mentioned in the introduction, we will avoid giving a detailed definition of full and strong dualizability. Instead we will utilize the following results from~\cite{CD98} and~\cite{LMW}.

\begin{definition}[\cite{LMW}]\label{D:enoughTAO}
A finite algebra $\bA$ \emph{has enough total algebraic operations}, if there exists $\varphi\colon\omega\to\omega$ such that for all $\bB \le \bC \le \bA^n$ and every $h \in \hom(\bB, \bA)$,
which has an extension to $\bC$, there exists $X \subseteq \hom(\bA^n,\bA)$ such that
\begin{enumerate}
\item $|X| \le \varphi(|B|)$,
\item There is a homomorphism $k$ from $\bC/\cap\setm{\ker(f|_C)}{ f \in X}$ to $\bA$ such that $k\circ\alpha=h$,  where $\alpha$ is the natural map from $\bB$ to  $\bC/\cap\setm{\ker(f|_C)}{f \in X}$.
\end{enumerate}
\end{definition}

\begin{theorem}[\cite{LMW}, Theorem 4.3]\label{T:enoughTAO}
A finite dualizable algebra that has enough total algebraic operations is strongly dualizable.
\end{theorem}

\begin{definition}[\cite{CD98}, pg. 73] Let $\bA$ be an algebra. The \textbf{\emph{enriched partial hom-clone}} of $\bA$ consists of all homomorphisms from $\bB$ to $\bA$, for all subalgebras $\bB$ of $\bA^n$, and
all positive integers $n$.
\end{definition}

\begin{theorem}[\cite{CD98}, Brute Force Strong Duality Theorem 3.2.2] \label{T:StBruteForce}
Let $\bA$ be a finite algebra. If some alter ego  $\undertilde{\bA}'$ yields a strong duality on $\bA$, then $\undertilde{\bA}=\langle A, \cP, \tau\rangle$,
yields a strong duality on $\bA$, where $\cP$ is the  enriched partial hom-clone of $\bA$ and $\tau $ is the discrete topology on $A$.
\end{theorem}

Our next result is a special application  of the M-shift strong duality Lemma from \cite{CD98} to the alter ego $\langle A, \cP, \tau\rangle$.

\begin{lemma}[cf. \cite{CD98}, Lemma 3.2.3] \label{L:StShift}Let $\bA$, $\cP$, and $\tau$ be as in the Theorem \ref{T:StBruteForce}. Let $\cP'\subseteq \cP$ be a generating set of $\cP$, that is, every $h \in \cP$ is a composition of elements of $\cP'$ and projections. If $\langle A, \cP, \tau\rangle$ yields a strong duality on $\bA$, then $\langle A, \cP', \tau\rangle$ yields a strong duality on $\bA$.
 \end{lemma}

\section{A generating set for domains of partial functions}\label{S:factoring}

In order to show our main result, we want to establish that every Abelian algebra satisfies the conditions of Definition \ref{D:enoughTAO}, so that we may use Theorem \ref{T:enoughTAO}. The set $X$ appearing in  the definition can actually be taken as a set of coordinate projections. Hence to establish a necessary bound on $X$, we need to be able to show that partial compatible functions on $\bA$ (i.e. homomorphisms from subpowers of $\bA$ to $\bA$), factor though partial compatible functions of  bounded arity. As a first step towards our result, in this section we show that we can generate all possible domains of such functions from a finite set.

The following definitions are from \cite{Gil14}. Given Abelian algebras $\bA$ and $\bS$ and a homomorphism $k\colon \bA \to \bS$, let $\cH_k(\bA^2, \bS)$ consist of all homomorphisms $f\colon \bA^2 \to \bS$ that satisfy $f(x,x)=k(x)$. We set $\bar k \in \cH_k(\bA^2, \bS)$ as $\bar k(x,y) = k(y)$. In \cite[Lemma~~5.4]{Gil14}, it is shown that $\langle\cH_k(\bA^2,\bS);+^{\bar k}\rangle$ is an Abelian group (where $+^{\bar k}$ is defined as in Lemma \ref{L:SeparerConstantes}), and that the  isomorphism type of $\langle\cH_k(\bA^2,\bS);+^{\bar k}\rangle$ does not depend on $k$. We let $\langle\cH(\bA^2,\bS);+\rangle$ stand for this isomorphism type.

The following lemma, proved in \cite[Lemma~5.7]{Gil14},  expresses that (total) homomorphisms $f\colon \bA^n\to \bS$ can be factored through a small power of $\bA$, which does not depend on $n$ but depends only on  $\langle\cH(\bA^2,\bS);+\rangle$.
\begin{lemma}\label{L:PetiteFactorisation}
Let $\bA,\bS$ be algebras in a variety of Abelian algebras. Let $N$ be a positive integer such that $\langle\cH(\bA^2,\bS);+\rangle$ has a family of generators with $N-1$ elements. Let $f\colon \bA^n\to \bS$ be a homomorphism. Then there exists a homomorphism $p\colon \bA^n\to \bA^{N}$ that is a term in $t$, and a homomorphism $q\colon \bA^{N}\to \bS$, such that $f=q\circ p$.
\end{lemma}
\begin{corollary}\cite[Corollary~5.6]{Gil14}\label{C:BoundH} Let $\bA$ and $\bS$ be Abelian algebras such that  $|A|=p_1^{\alpha_1}\dots p_k^{\alpha_k}$, $|S|=p_1^{\beta_1}\dots p_k^{\beta_k}$ for distinct primes $p_i$. Then
$|\cH(\bA^2, \bS)|$ divides $p_1^{\alpha_1\beta_1}\dots p_k^{\alpha_k \beta_k}$, and $\langle\cH(\bA^2,\bS);+\rangle$ has a generating set of size $\max_{1 \le i\le k}(\alpha_i\beta_i)$.
\end{corollary}
\begin{corollary}\label{C:CloneDomaineFinimentEngendre}
Let $\bA$ be a finite Abelian algebra. Let $p_1^{\alpha_1}p_2^{\alpha_2}\dots p_k^{\alpha_k}$ be the prime decomposition of $\card A$.
Let $N=1+\max_{1\le i \le k}(\alpha_i^3)$. Denote by $\cF$ the partial clone over $A$ generated by $t$ and all $\pi_1^N\res C$ for $\bC$ a subalgebra of $\bA^N$. Then $\pi_1^n\res D \in\cF$ for all positive integer $n$ and all subalgebras $\bD$ of $\bA^n$.
\end{corollary}
\begin{proof}
Let $n$ be a positive integer, let $\bD$ be a completely meet-irreducible subalgebra of $\bA^n$. Set $\bS=\bA^n/\Theta_{\bD}$, and denote by $f\colon\bA^n\to\bS$ the canonical projection.

By Lemma \ref{L:ThetaDiscr}, $\set D$ is the underlying set of a (one-element) subalgebra of $\bS$, moreover $x\in  D\Longleftrightarrow f(x)=D$. That is $f^{-1}(\set D)=D$.

As $\bD$ is completely meet-irreducible, by Lemma \ref{L:MeetIrred}, $\bS$ is subdirectly irreducible. By Theorem \ref{T:countSubIrred}, $\card S$ divides $p_1^{\alpha_1^2}\dots p_k^{\alpha_k^2}$.
 By Corollary \ref{C:BoundH} the group $\cH(\bA^2,\bS)$ has a generating family with $\max_{1\le i\le k}(\alpha_i^3)=N-1$ elements.

Therefore, by Lemma~\ref{L:PetiteFactorisation}, there exists a homomorphism $p\colon\bA^n\to\bA^N$, which is a term in $t$, and a homomorphism $q\colon\bA^N\to\bS$ such that $q\circ p=f$.

Set $C=q^{-1}(\set{D})=\setm{x\in A^N}{q(x)=D}$. As $\set D$ is the underlying set of a subalgebra of $\bS$ and $q$ is a homomorphism, it follows that $C$ is the underlying set of a subalgebra of $\bA^N$, hence $C$ is a domain of $\cF$. Moreover $p$ is a term of $t$, thus it follows from Lemma~\ref{L:ClonePartielDomaine}(4) that $p^{-1}(C)$ is a domain of $\cF$. However $p^{-1}(C)=p^{-1}(q^{-1}(\set D))=f^{-1}(\set D)=D$, so $D$ is a domain of $\cF$.

Let $\bB$ be an arbitrary subalgebra of $\bA^n$. We can write $B$ as the intersection of finitely many underlying sets of completely subdirectly irreducible subalgebras of $\bA^n$. Since each of these sets is a domain of $\cF$, it follows from Lemma~\ref{L:ClonePartielDomaine}(2) that $B$ is a domain of $\cF$. Therefore, by Lemma~\ref{L:ClonePartielDomaine}(3), $\pi_1^n\res B$ is in $\cF$.
\end{proof}

\section{Extensions of Partial homomorphisms} \label{S:Extensions}

The results of Corollary \ref{C:CloneDomaineFinimentEngendre} imply that we may generate a partial homomorphism from its extension to a larger domain and a bounded number of partial projections.
Thus, the goal of this section is to extend partial homomorphisms of Abelian algebras (in a finitely generated variety of Abelian algebras). We will show that if the domain of a partial homomorphism is small enough, then the partial homomorphism has a proper extension (cf. Lemma~\ref{L:ExtensionsMorphismesAffine}). We will first establish this result for modules before generalizing to Abelian algebras.

\begin{lemma}\label{L:ExtensionPairesMorphismes}
Let $\bB,\bC$ be submodules of a module $\bA$. Let $\bE$ be a module. Let $f\colon \bB\to \bE$ and $g\colon \bC\to \bE$ be homomorphisms. If $f\res B\cap C=g\res B\cap C$, then there exists a homomorphism $h\colon \bB+\bC\to \bE$ that is a common extension of $f$ and $g$.
\end{lemma}

\begin{proof}
Let $b,b'\in B$ and $c,c'\in C$. Assume that $b+c=b'+c'$. Then $b-b'=c'-c$ belongs to $B\cap C$, hence
\begin{equation*}
f(b)-f(b')=f(b-b') = g(b-b')=g(c'-c)=g(c')-g(c)\,.
\end{equation*}
Therefore $f(b) + g(c) = f(b')+ g(c')$. It follows that the map
\begin{align*}
h\colon B+C &\to E\\
b+c & \mapsto f(b)+g(c)\,,
\end{align*}
is well-defined, and is a homomorphism of modules. Moreover $h(b)=h(b+0)=f(b)+g(0)=f(b)$ for all $b\in B$. Similarly $h$ extends $g$.
\end{proof}

\begin{lemma}\label{L:BorneUniformeSousAlgebres}
Let $\cV$ be a locally finite variety of algebras. Let $\bA\in\cV$. Let $n$ be a positive integer. If each finitely generated subalgebra of $\bA$ is generated by $n$ elements, then $\bA$ is finite.
\end{lemma}

\begin{proof}
Assuming that $\bA$ is infinite, there is an infinite sequence $(x_i)_{i\in\mathbb{N}}$ of distinct elements of $\bA$. Let  $k = \card{\bF_\cV(n)}$.

Denote by $\bB$ the subalgebra of $\bA$ generated by $\set{x_0,x_1,x_2,\dots,x_k}$. Note that $\card B\ge k+1$, but $\bB$ is finitely generated, so is generated by $n$ elements, hence $\card B\le\card{F_\cV(n)}=k$, a contradiction.
\end{proof}

\begin{lemma}\label{L:famillegeneratrice}
Let $\bA$ be an Abelian group such that $\card A=p_1^{\alpha_1}\dots p_k^{\alpha_k}$ where $p_1,\dots,p_k$ are distinct primes. Set $N=1+\alpha_1+\dots+\alpha_k$. Let $a_1,\dots,a_N \in A$. Then there are  integers $u_1,\dots,u_{i-1}$ for  some $i$ with $1\le i\le N $, such that $a_i=\sum_{j=1}^{i-1} u_ja_j$.
\end{lemma}

\begin{proof}
Set $A_0=\set 0$. Given $1\le i\le N$, denote by $\bA_i$ the subgroup of $\bA$ generated by $\set{a_1,\dots,a_i}$. Note that $\bA_i$ is a subgroup of $\bA_j$ for $0\le i\le j\le N$.

As a maximal chain of subgroups of $\bA$ has size at most $N$, it follows that there is $1\le i\le N$ such that $A_i=A_{i-1}$. Therefore $a_i\in A_{i-1}$, so there are integers $u_1,\dots,u_{i-1}$ such that $a_i=\sum_{j=1}^{i-1} u_ja_j$.
\end{proof}

\begin{lemma}\label{L:ExtensionsMorphismesModulesI}
Let $\cV$ be a locally finite variety of $\bR$-modules, and $\bE \in \cV$ finite. Then there exists a positive integer $N$ such that, given modules $\bB\subseteq \bC$ in $\cV$ and a homomorphism $f\colon \bB\to \bE$, if $\bC/\bB$ is not generated by $N$ elements, then $f$ has a proper extension.

Moreover, let $\card R=p_1^{r_1}\dots p_k^{r_k}$, and $\card E=p_1^{\beta_1}\dots p_k^{\beta_k}$, where the $p_i$ are distinct primes, and assume that $\bR$ as an $\bR$-module has $\ell$ strict non-trivial submodules. Then we can pick $N = \ell\times \left(\sum_{i=1}^k r_i\beta_i -1\right)$.
\end{lemma}

\begin{proof}
Denote by $\bF$ the $\bR$-module freely generated by $\set{u}$, so that $\bF\cong \bR$ as $\bR$-modules.

 Given a strict submodule $\bG$ of $\bF$, denote by $N_{\bG}$ a positive integer with the following property: given $\varphi_1,\dots,\varphi_{N_{\bG}}\in \Hom(\bG,\bE)$, there are integers $u_1,\dots,u_{i-1}$ for some $i$ with $1\le i\le N_{\bG}$, such that $\varphi_i=\sum_{j=1}^{i-1} u_j\varphi_j$.
As $\bG$ is a strict sub-module of $\bF$, it follows from Lemma~\ref{L:Majorations} that $\card{\Hom(\bG,\bE)}$ divides $p_1^{r_1\beta_1}\dots p_k^{r_k\beta_k}$ strictly, hence by Lemma~\ref{L:famillegeneratrice}, we may choose $N_G=\sum_{i=1}^{k} r_i\beta_i$.
Set
\begin{equation}\label{eq:N}
N=\sum\famm{N_{\bG}-1}{\bG\text{ is a nontrivial strict submodule of } \bF}\,.
\end{equation}
Note that $N\le \ell\times (\sum_{i=1}^{k} r_i\beta_i -1)$.

Let $\bB\subseteq \bC$ in $\cV$ and $f\colon \bB\to \bE$ be a homomorphism. Assume that $\bC/B$ is not generated by $N$ elements.

First note that if all finitely generated submodules of $\bC/B$ are generated by $N$ elements, then it follows from Lemma~\ref{L:BorneUniformeSousAlgebres} that $\bC/B$ is generated by $N$ elements, which contradicts the assumption. Therefore there is a finitely generated submodule $\bQ$ of $\bC/B$, whose minimal number of generators is $k\ge N+1$.

Let $\bP$ the submodule of $\bC$ containing $\bB$ such that $\bP/B=\bQ$. Note that $\set{x_1+B,x_2+B,\dots,x_k+B}$ generates~$\bQ$ if and only if $B\cup\set{x_1,\dots,x_k}$ generates~$\bP$. We say that $x_1,\dots,x_k$ \emph{generate $\bP$ over~$\bB$}.

Given $x\in C$ we denote by $\varphi_x\colon \bF\to \bC$ the unique homomorphism that maps $u$ to $x$. Note that $\varphi_x+\varphi_y=\varphi_{x+y}$ for all $x,y\in C$.

Pick $x_1,\dots,x_k \in P \setminus B$, generating $\bP$ over $\bB$, such that $(\varphi_{x_i}^{-1}(B))_{1\le i\le k}$ is maximal. That is, if $y_1,\dots,y_k$ generate $\bP$ over $\bB$ and $\varphi_{y_i}^{-1}(B)\supseteq \varphi_{x_i}^{-1}(B)$ for all $1\le i\le k$, then $\varphi_{y_i}^{-1}(B)=\varphi_{x_i}^{-1}(B)$ for all $1\le i\le k$. The existence of such a sequence follows from the finiteness of $\Sub\bF$.

Set $S_i=\varphi_{x_i}^{-1}(B)$. That is, $\bS_i$ is the largest submodule of $\bF$ such that $\varphi_{x_i}(S_i)\subseteq B$, for each $1\le i\le k$. Note that $\varphi_{x_i}(F)\cap B=\varphi_{x_i}(S_i)$.

Assume that $S_i=\set 0$ for some $1\le i\le k$. Then $\varphi_{x_i}(F)\cap B=\set 0$. Let $g\colon\varphi_{x_i}(F)\to\bE$, $x\mapsto 0$. Note that $\dom f\cap \dom g=\set{0}$, hence it follows from Lemma~\ref{L:ExtensionPairesMorphismes} that there exist a morphism $h\colon B+\varphi_{x_i}(F)$ that extends both $f$ and $g$. As $x_i\not\in B$, it follows that $h$ is a proper extension of $f$.

We now assume that $S_i\not=\set{0}$ for each $1\le i\le k$.

\begin{claim}
The $\bS_i$ are strict submodules of $\bF$, for all $1\le i\le n$.
\end{claim}

\begin{cproof}
Assume we have $i$ such that $S_i=F$. Hence $u$, the generator of $\bF$, belongs to $\bS_i$, so $x_i=\varphi_{x_i}(u)\in \varphi_{x_i}(S_i)\subseteq B$, contradicting that $x_i \notin B$.
\end{cproof}

As $k\ge N+1>N$ and the $\bS_i$ are strict nontrivial submodules of $\bF$, it follows  that there are a submodule $\bG$ of $\bF$, $\bG$ strict and non-trivial, and $I\subseteq\set{1,\dots,k}$, such that $\card I=N_{\bG}$ and $S_i=G$, for all $i\in I$.

Set $\psi_i=f\circ\varphi_{x_i}\res G$. As $\card I=N_{\bG}$, there is $i\in I$
and a family of integers $(u_j)_{j\in J}$ (where $J=I\setminus\set i$) such that
\begin{equation*}
\psi_i=\sum_{j\in J}u_j\psi_j.
\end{equation*}
Let $y=x_i-\sum_{j\in J}u_j x_j$. Note that $x_1,\dots,x_{i-1},y,x_{i+1},\dots, x_k$ generates $\bP$ over $\bB$. Moreover
\begin{equation*}
\varphi_y(s)=\varphi_{x_i-\sum_{j\in J}u_j x_j}(s)=\varphi_{x_i}(s) - \sum_{j\in J} u_j\varphi_{x_j}(s)\in B\,,\quad\text{for all $s\in G$}
\end{equation*}
thus $\varphi_{y}^{-1}(B)\supseteq G=S_i$. It follows from the maximality of $(S_1,\dots,S_k)$ that $\varphi_{y}^{-1}(B)=S_i=G$.

Let $z\in \varphi_{y}(F)\cap B$, and take $s\in G$ such that $\varphi_{y}(s)=z$. The following equalities hold
\begin{align*}
f(z)&=f(\varphi_{y}(s))\\
&=f\left(\varphi_{x_i}(s)-  \sum_{j\in J}u_j\varphi_{x_j}(s)\right)\\
&=f(\varphi_{x_i}(s))- \sum_{j\in J}u_jf(\varphi_{x_j}(s))\\
&=\psi_i(s)-\sum_{j\in J}u_j\psi_j(s)\\
&=0
\end{align*}

Denote by $g\colon \varphi_{y}(F)\to E$ the constant 0 morphism. We have
\begin{equation*}
f\res \varphi_{y}(F)\cap B = 0= g\res \varphi_{y}(F)\cap B\,.
\end{equation*}
Set $D=B+\varphi_{y}(F)$. It follows from Lemma~\ref{L:ExtensionPairesMorphismes} that $f$ and $g$ have a common extension $h\colon \bD\to \bE$.

Note that $y=\varphi_{y}(u)\in D$ and $y\not\in B$, so $h$ is a strict extension of $f$.
\end{proof}

\begin{lemma}\label{L:ExtensionsMorphismesAffine}
Let $\cV$ be a locally finite variety of affine algebras. Let $\bE$ be a finite algebra in $\cV$. Then there exists a positive integer $N$ such that, given algebras $\bB\subseteq \bC$ in $\cV$ and a homomorphism $f\colon \bB\to \bE$, if
$\bC/\Theta_B$ is not generated by $N$ elements, then $f$ has a proper extension.
\end{lemma}

\begin{proof}
Denote by $\sL$ the similarity type of $\cV$. Let $c$ be a constant symbol that is not in $\sL$.

We consider $\cV_c$, as defined in Lemma~\ref{L:AjouterConstantes}. Note that $\cV_c$ is a locally finite variety of affine algebras, moreover there is a constant operation. It follows from Lemma~\ref{L:SeparerConstantes} that there is a similarity type $\sL'\cup\sC$ satisfying (1)-(5) of Lemma~\ref{L:SeparerConstantes}.

We denote by $\cW$ the class of all reducts of algebras in $\cV_c$, to the type $\sL'$. So $\cW$ is a variety of modules.

Let $E$ be a finite algebra in $\cV$. Given $y\in E$, we consider its reduct $E_y(\sL')\in\cW$ and take $N_y$ as in (\ref{eq:N}), with regard to the variety $\cW$. Set $N=1+\max\setm{N_y}{y\in E}$.

Let $B\subseteq C$ in $\cV$. Let $f\colon B\to E$ be a homomorphism. Assume that $C/\Theta_B$ is not generated by $N$ elements. Pick $x\in B$. Note that $(C/\Theta_B)_{[x]_{\Theta_B}}(\sL\cup\set c)$ is not generated by $N-1$ elements. It follows that $C_x(\sL')/B_x(\sL')=(C/B)_{x/B}(\sL')$ is not generated by $N-1$ elements.

The map $f\colon B_x\to E_{f(x)}$ is a homomorphism of $\sL\cup\set{c}$ algebras, and so is a homomorphism of $\sL'\cup\sC$-algebras. However, since $N-1\ge N_{f(x)}$, there exists $h\colon D(\sL')\to E_{f(x)}(\sL')$ where $D(\sL')$ is a subalgebra of $C_x(\sL')$ properly containing $B$, and so $h$ is a homomorphism of $\sL'$-algebras that strictly extends $f$.

Since $B$ contains all the constants operations in $\sC$,  $D(\sL')$ contains the constants. Hence we can consider $D(\sL'\cup\sC)$; moreover, as $h$ extends $f$, it follows that $h$ is a morphism of $\sL'\cup\sC$-algebras, and so is a homomorphism of $\sL$-algebras. Therefore $f$ has a proper extension (as homomorphism of $\sL$-algebras).
\end{proof}

The following corollary is an immediate consequence of Lemma~\ref{L:ExtensionsMorphismesAffine}. Informally the domain of a non-extensible homomorphism of Abelian algebras is large.
\begin{corollary}\label{C:MorphismesMaximaux}
Let $\cV$ be a locally finite variety of affine algebras. Let $E$ be a finite algebra in $\cV$. Then there exists a positive integer $N$ such that, given algebras $B\subseteq C$ in $\cV$ and $f\colon B\to E$, if $f$ has no proper extension, then $C/\Theta_B$ is generated by $N$ elements, and so is finite.
\end{corollary}

\section{Factoring partial homomorphisms} \label{S:factorPartial}

The main goal of this section is to factorize a partial homomorphism $f\colon \bC\to\bE$ (where $\bC\subseteq\bA^n$) through a smaller power $\bD\subseteq\bA^N$, where $N$ only depends on $\bA$ and $\bE$. This will allow us to use Theorem \ref{T:enoughTAO} and to prove our main result.

First note that Abelian algebras have the congruence extension property. To be more precise we give the following description of extensions of congruences.

\begin{lemma}\label{L:ExtensionCongruences}
Let $\bA$ be a subalgebra of an Abelian algebra $\bB$. Let $\alpha$ be a congruence of $\bA$. Then there exists a smallest extension of $\alpha$ to $\bB$. It is the unique congruence $\beta$ of $\bB$ satisfying the following conditions.
\begin{enumerate}
\item For all $(x,y)\in\beta$, if $y\in A$, then $x\in A$.
\item $\beta\cap A^2=\alpha$.
\end{enumerate}
\end{lemma}

\begin{proof}
We can assume that the neutral element of $+$ belongs to $\bA$. Indeed we pick $0\in A$ and set $x+y=t(x,0,y)$. It follows that $A$ is stable for $+$. Also note that $x-y=t(x,y,0)$ and $t(x,y,z)=x-y+z$.

Let $(x,x')\in\alpha$, $(y,y')\in\alpha$. As $\alpha$ is compatible with $t$, it follows that $(x+y,x'+y') = ( t(x,0,y),t(x',0,y'))\in\alpha$. Therefore $\alpha$ is compatible with $+$ and $-$.

Define $\beta=\setm{(x,y)\in B^2}{(\exists a\in B)( (x-a,y-a)\in\alpha)}$. We will leave it to the reader to check that $\beta$ is a congruence that satisfies conditions  (1) and (2).

Let $\gamma$ be a congruence of $\bB$ containing $\alpha$. Let $x,y,a\in B$ such that $(x-a,y-a)\in\alpha$, so $(x-a,y-a)\in\gamma$, hence $(x,y)=(x-a+a,y-a+a)\in\gamma$. Therefore $\gamma$ contains $\beta$, hence $\beta$ is the smallest extension of $\alpha$ to $\bB$.

Let $\gamma$ be a congruence satisfying $(1)$ and $(2)$. First note that $\gamma$ contains $\alpha$, so $\gamma$ contains $\beta$. Conversely, let $(x,y)\in\gamma$. Note that $(x-y,0)=(x-y,y-y)\in\gamma$ and $0\in A$, so $x-y\in A$, and so $(x-y,y-y)\in \gamma\cap A^2=\alpha$, that is $(x,y)\in\beta$. Therefore $\gamma=\beta$.
\end{proof}

\begin{lemma}\label{L:CardinalDoubleQuotient}
Let $\bA\subseteq\bB$ be Abelian algebras. Let $\alpha$ be a congruence of $\bA$. Let $\beta$ be the minimal extension of $\alpha$ to $\bB$. Then $\card{\bB/\beta} = \card{\bB/\Theta_A} \times \card{\bA/\alpha}$.
\end{lemma}

\begin{proof}
First note that $\Theta_{\bA}\supseteq\alpha$, hence $\Theta_{\bA}\supseteq\beta$. Also note that $\bA/\beta$ is a subalgebra of $\bB/\beta$.

Let $(x/\beta,y/\beta)\in\Theta_{\bA/\beta}$. There is $u\in \bA/\beta$ such that $y/\beta-x/\beta + u\in \bA/\beta$, hence there is $c\in A$ such that $(y-x+c)/\beta\in \bA/\beta$. That is $(y-x+c)/\beta=a/\beta$ for some $a\in A$. It follows from Lemma~\ref{L:ExtensionCongruences}(1) that $(y-x+c)\in A$. Therefore $(x,y)\in \Theta_{\bA}$, hence $(x/\beta,y/\beta)\in\Theta_{\bA}/\beta$.

Conversely, let $(x/\beta,y/\beta)\in\Theta_{\bA}/\beta$. That is $(x,y)\in\Theta_{\bA}$. There is $c\in A$ such that $x-y+c\in A$, so $x/\beta-y/\beta+c/\beta\in A/\beta$. As $c/\beta\in \bA/\beta$, it follows that $(x/\beta,y/\beta)\in\Theta_{\bA/\beta}$. This proves that $\Theta_{\bA/\beta}=\Theta_{\bA}/\beta$.

Note that $\bB/\Theta_{\bA}\cong (\bB/\beta)/(\Theta_{\bA}/\beta) = (\bB/\beta)/(\Theta_{\bA/\beta})$. Lemma~\ref{L:CardinalQuotient} implies that $\card{\bB/\Theta_{\bA}} = \card{\bB/\beta}/\card{\bA/\beta}$. Now $\beta\cap A^2=\alpha$, so $\bA/\beta\cong\bA/\alpha$, and therefore $\card{\bB/\beta} = \card{\bB/\Theta_{\bA}} \times \card{\bA/\beta}$.
\end{proof}

\begin{theorem}\label{T:FactorisationPetitePuissance}
Let $\bA$ and $\bE$ be finite in a locally finite variety of Abelian algebras. Then there exists a positive integer $\ell$ such that, given a positive integer $n$, a subalgebra $\bC$ of $\bA^n$ and a homomorphism $h\colon \bC\to\bE$, there exists a homomorphism $p\colon \bA^n\to\bA^\ell$ and a homomorphism $k\colon p(\bC)\to \bE$ such that $k\circ p\res C=h$. Moreover we can choose $p$ to be a term in $t$.
\end{theorem}

\begin{proof}
Denote by $\cV$ a locally finite variety of Abelian algebras containing $\bA$ and $\bE$. Note that most of the homomorphisms and commuting relations used in this proof are illustrated in Figure~\ref{F:DiagrammePrincipal}.

Let $N$ be as in Corollary~\ref{C:MorphismesMaximaux}. Set $N'=\card{F_{\cV}(N)}\times\card{A}$.

Given $\bS\in\cV$, such that $\card{\bS}$ divides $N'$, we pick $\ell_{\bS}$ such that $\langle\cH(\bA^2,\bS);+\rangle$ has a generating family with $\ell_{\bS}$ elements. Set
\begin{equation*}
\ell=1+\max\setm{\ell_{\bS}}{\bS\in\cV\text{ and }\card{\bS}\text{ divides } N'}\,.
\end{equation*}

Let $n$ be a positive integer, $\bC \in \Sub(\bA^n)$, and $h\colon\bC\to\bE$ a homomorphism. Let $h'\colon\bD\to\bE$ be a maximal extension of $h$. Denote by $\eta_1\colon\bC\to\bD$ the inclusion homomorphism. As $h'$ extends $h$, we have
\begin{equation}\label{E:p2}
h'\circ\eta_1=h\,.
\end{equation}

It follows from Corollary~\ref{C:MorphismesMaximaux} that $\bA^n/\bD$ is generated by $N$ elements, hence $\card{\bA^n/\bD}$ divides $\card{F_{\cV}(N)}$. Denote by $\varepsilon_1\colon\bD\to\bA^n$ the inclusion morphism.

Denote by $\alpha$ the kernel of $h'$ and by $\beta$ the minimal extension of $\alpha$ to $\bA^n$. Set $\bR=\bD/\alpha$ and $\bS=\bA^n/\beta$. Let $\varepsilon_2\colon\bR\to\bS$ be the canonical embedding. Let $\pi\colon\bD\to\bR$ and $\pi'\colon\bA^n\to\bS$ be the canonical projections. Note that
\begin{equation}\label{E:p3}
\pi'\circ\varepsilon_1=\varepsilon_2\circ\pi\,.
\end{equation}

Note that $h'$ factors through $\pi$, so there is an embedding $\sigma\colon\bR\to \bA$ such that
\begin{equation}\label{E:p1}
h'=\sigma\circ\pi\,.
\end{equation}
Hence $\card \bR = \card{\bD/\alpha}$ divides $\card \bA$. It follows from Lemma~\ref{L:CardinalDoubleQuotient} that
$$\card\bS=\card{\bA^n/\beta}=\card{\bA^n/\Theta_{\bD}}\times\card{\bD/\alpha}\text{ divides } N'\,,$$
hence $1+\ell_{\bS}\le \ell$. That is $\langle\cH(\bA^2,\bS);+\rangle$ has a generating family with $\ell-1$ elements.

From Lemma~\ref{L:PetiteFactorisation}, we have morphisms $p\colon\bA^n\to\bA^\ell$ and $q\colon\bA^\ell\to\bS$ such that
\begin{equation}\label{E:p5}
q\circ p=\pi'\,.
\end{equation}

As $\bD$ is a subalgebra of $\bA^n$, it follows that $p(\bD)$ is a subalgebra of $\bA^\ell$. Denote by $\varepsilon_3\colon p(\bD)\to\bA^\ell$ the canonical embedding. Note that
\begin{equation}\label{E:p6}
\varepsilon_3\circ p\res\bD=p\circ\varepsilon_1\,.
\end{equation}
Similarly we denote by $\eta_2\colon p(\bC)\to p(\bD)$ the inclusion morphism, so
\begin{equation}\label{E:p4}
\eta_2\circ p\res\bC=p\res\bD\circ\eta_1\,.
\end{equation}

The following equalities are direct consequences of \eqref{E:p3}, \eqref{E:p5}, and \eqref{E:p6}
\begin{equation}\label{E:qe3p=e2pi}
q\circ\varepsilon_3\circ p\res\bD=q\circ p\circ\varepsilon_1=\pi'\circ\varepsilon_1=\varepsilon_2\circ\pi\,.
\end{equation}
So $q(\varepsilon_3(p(D)))=\varepsilon_2(\pi(D))=\varepsilon_2(R)$. However, $\varepsilon_2$ is an embedding, so $q(\varepsilon_3(p(D)))$ corresponds to a subalgebra of $R$; hence, there is a homomorphism $u\colon p(\bD)\to\bR$ such that
\begin{equation}\label{E:p7}
q\circ\varepsilon_3=\varepsilon_2\circ u\,.
\end{equation}

It follows from \eqref{E:p7} and \eqref{E:qe3p=e2pi} that $\varepsilon_2\circ u\circ p\res D= q\circ\varepsilon_3\circ p\res\bD=\varepsilon_2\circ\pi$. As $\varepsilon_2$ is an embedding it follows that
\begin{equation}\label{E:up=pi}
u\circ p\res\bD=\pi\,.
\end{equation}

The following equalities hold
\begin{align*}
\sigma\circ u\circ\eta_2\circ p\res\bC &= \sigma\circ u\circ p\res\bD\circ\eta_1 & \text{by \eqref{E:p4}.}\\
& =\sigma\circ\pi\circ\eta_1 &  \text{by \eqref{E:up=pi}.} \\
&=h'\circ\eta_1 & \text{by \eqref{E:p1}.}\\
&=h & \text{by \eqref{E:p2}.}\\
\end{align*}
Therefore, denoting $k=\sigma\circ u\circ\eta_2\colon p(\bC)\to\bE$, we have $k\circ p\res\bC=h$.
\end{proof}

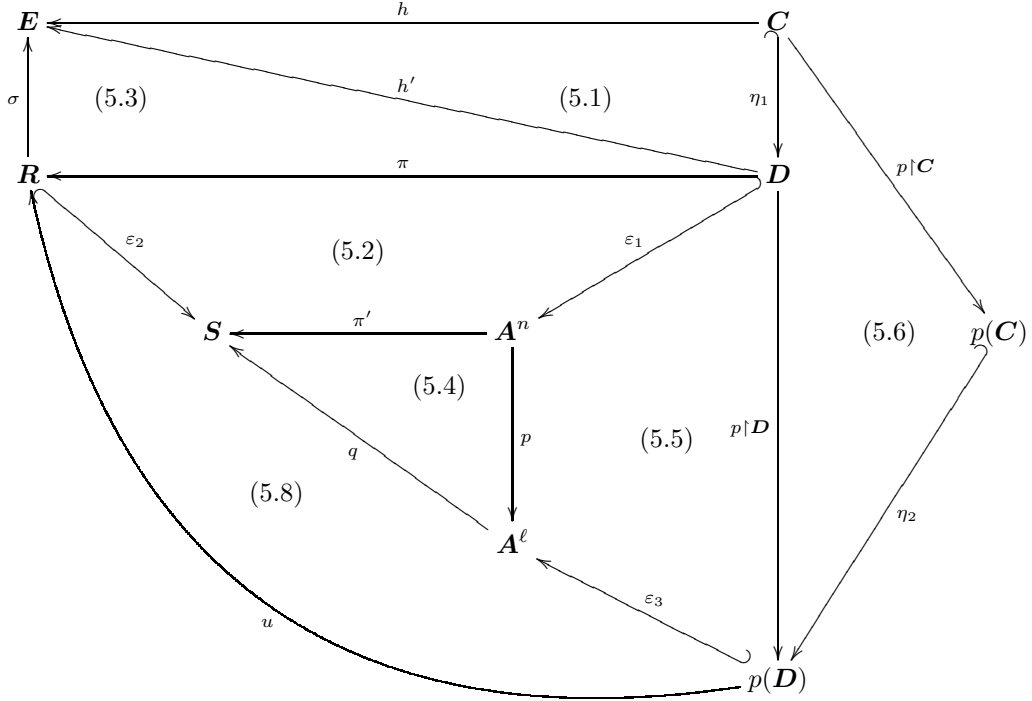
\begin{figure}
\xymatrix@R=0.15cm@C=0.15cm{
\bE & & & & & & & & & & & & \bC \ar[llllllllllll]_h  \ar@{_{(}->}[dddd]_{\eta_1} \ar[ddddddddrrrr]^{p\res\bC} \\
\\
& & \eqref{E:p1} & & & & & & & \eqref{E:p2}\\
\\
\bR  \ar[uuuu]^\sigma \ar@{_{(}->}[ddddrrrr]^{\varepsilon_2} & & & & & & & & & & & & \bD \ar[lllllllllllluuuu]_{h'} \ar[llllllllllll]_\pi \ar@{_{(}->}[lllldddd]_{\varepsilon_1} \ar[dddddddddddd]_{p\res\bD}\\
\\
& & & & & & \eqref{E:p3} \\
\\
& & & & \bS & & & & \bA^n \ar[llll]_{\pi'} \ar[dddd]^{p}  & & & & & & \eqref{E:p4} & & p(\bC)  \ar@{_{(}->}[lllldddddddd]^{\eta_2} \\
& & & & & & & \eqref{E:p5}\\
& & & & & & & &  & & \eqref{E:p6} \\
& & & & & \eqref{E:p7} \\
& & & & & & & & \bA^\ell \ar[lllluuuu]^{q} \\
 \\
 \\
 \\
 & & & & & &  & & & & & & p(\bD) \ar@{_{(}->}[lllluuuu]_{\varepsilon_3} \ar@/^7pc/[lllllllllllluuuuuuuuuuuu]^u
}
\caption{Homomorphisms factor through sub-power of small dimension.}\label{F:DiagrammePrincipal}
\end{figure}

\begin{corollary}\label{C:ClonePartielFinimentEngendre}
Let $\bA$ be a finite Abelian algebra. The partial clone of partial operations on $A$ compatible with $\bA$ is finitely generated.
It is generated by partial operations of arity at most $\max \set{3,\ell,N}$, where $N$ is the bound from Corollary~\textup{\ref{C:CloneDomaineFinimentEngendre}}, and $\ell$ is the bound from Theorem~\textup{\ref{T:FactorisationPetitePuissance}} in the case that $\bE=\bA$.
\end{corollary}

\begin{proof}
Given integers $1\le i\le n$ , we denote by $\pi^n_i\colon\bA^n\to\bA$ the canonical projection on the $i$-th coordinate.

Let $\ell$ be as in Theorem~\ref{T:FactorisationPetitePuissance} for $\bE=\bA$. Let $N$ be as in Corollary~\ref{C:CloneDomaineFinimentEngendre}. Set $X=\setm{f\colon \bD\to\bA}{\bD\text{ is a subalgebra of }\bA^\ell}$. Set $Y=\setm{\pi_1^N\res C}{\bC\text{ is a subalgebra of }\bA^N}$.

Denote by $\cF$ the partial clone generated by $\set t\cup X\cup Y$.

Let $n$ be a positive integer, let $\bC$ be a subalgebra of $\bA^n$, let $h\colon\bC\to\bA$ be a homomorphism. By Theorem \ref{T:FactorisationPetitePuissance}, there is a term $p\colon\bA^n\to\bA^\ell$ in $t$, and a homomorphism $k\colon p(\bC)\to\bA$ such that $h=k\circ p\res C$.

Note that $k\in\cF$, as it is a partial homomorphism of arity $\ell$, so $k\circ p$ belongs to $\cF$. Moreover, by Corollary~\ref{C:CloneDomaineFinimentEngendre}, $\pi_1^n|_C \in \cF$, therefore, by Lemma~\ref{L:ClonePartielDomaine}(3), $h=k\circ p\res C$ belongs to $\cF$.

Therefore $\cF$ is the set of all partial operations on $A$, compatible with $\bA$. Moreover $\cF$ is, by construction, finitely generated. The arity bound follows by construction as well.
\end{proof}

We remark in passing that our corollary provide an additional proof that every finite Abelian algebra $\bA$ is dualizable (even though it is unnecessarily complicated compared to the arguments in \cite{Gil14}). Dualizability of $\bA$ follows as by the Duality Compactness Theorem \cite{Wi:NTFP,Za:NDVA}, it suffices to show that the enriched partial hom-clone of $\bA$ is finitely generated.
Our arguments are however not independent, as we rely on Lemmas \ref{L:ThetaDiscr}, \ref{L:MeetIrred}, and \ref{L:PetiteFactorisation}  from \cite{Gil14}.

We are now ready to prove our final result about the strong dualizability of Abelian algebras.

\begin{lemma}\label{L:AATAO}
An Abelian algebra has enough total algebraic operations.
\end{lemma}

\begin{proof}
Let $\bA$ be an Abelian algebra, let $\ell$ be as in Theorem~\ref{T:FactorisationPetitePuissance} for $\bA=\bE$. We consider $\varphi\colon\omega\to\omega$ the constant map equal to $\ell$.

Let $n$ be a positive integer, and let $\bB\subseteq\bC$ be subalgebras of $\bA^n$. Denote by $\iota\colon\bB\to\bC$ the inclusion homomorphism. Let $h\colon\bB\to\bA$, let $h'\colon\bC\to\bA$ be an extension of $h$, that is $h'\circ\iota=h$.

By Theorem~\ref{T:FactorisationPetitePuissance} there exist a homomorphism $p\colon\bA^n\to\bA^\ell$ and a homomorphism $k\colon p(\bC)\to\bA$ such that $h'=k\circ p\res C$. We obtain $k\circ p\res C\circ\iota=h'\circ\iota=h$.
The result now follows with $X=\setm{\pi_i\circ p}{ 1\le i\le\ell}$.
\end{proof}

\begin{theorem}
Finite Abelian algebras are strongly dualizable.
\end{theorem}

\begin{proof}
Let $\bA$ be a finite Abelian algebra. By Lemma~\ref{L:AATAO}, $\bA$ has enough total algebraic operations, moreover $\bA$ is dualizable (\cite{Gil14}, see also the remark after Corollary \ref{C:ClonePartielFinimentEngendre}).
By \cite[Theorem 4.3]{LMW}, $\bA$ is strongly dualizable.
\end{proof}

Our main Theorem \ref{T:main} now follows from the well known fact that any strongly dualizable algebra is fully dualizable (see for example \cite{CD98}, Theorem 3.2.4). In our final result, we  provide
an explicit bound on the partial functions in the strongly dualizing alter ego.

\begin{theorem}\label{T:boundstrdual}
Let $\bA$ be a finite Abelian algebra with $\card A=p_1^{\alpha_1}\dots p_k^{\alpha_k}$. Then $A$ is strongly dualized by $\langle A;P, \tau\rangle,$ where $\tau$ is the discrete topology on $A$ and $P$ is the set of all algebraic partial operations on $A$ of arity at most
$$
1+\max_{1\le i\le k} \left(N\alpha_i^3+ 2\alpha_i^2 \right)\,,
$$
where $N=1+p_1^{\alpha_1^4}\dots p_k^{\alpha_k^4}\times \left(\sum_{i=1}^{k} \alpha_i^3 -1\right)$.
\end{theorem}

\begin{proof}
We may assume that $\bA$ is non-trivial. As $\bA$ is strongly dualizable, it is in particular strongly dualized by the strong brute force alter ego $\langle A; \cP, \tau \rangle$ of Theorem \ref{T:StBruteForce}.
Moreover, by Lemma \ref{L:StShift} we may replace $\cP$ with a generating set, and
by Corollary~\ref{C:ClonePartielFinimentEngendre}, $\cP$ is finitely generated. Hence $\bA$ is finitely strongly dualizable. It remains to establish the bound.

Corollary~\ref{C:ClonePartielFinimentEngendre} gives an arity bound of $\max \set{3,\ell,\bar N}$, where $\bar N$ is the bound from Corollary~\ref{C:CloneDomaineFinimentEngendre} (there
called $N$), and $\ell$ is the bound from Theorem~\ref{T:FactorisationPetitePuissance} in the case that $\bE=\bA$.

By Corollary~\ref{C:CloneDomaineFinimentEngendre}, $\bar N=1+\max_{1\le i\le k}^k(\alpha_i^3)$. Thus $ \bar N$ and $3$ are clearly smaller than the bound from the statement of the theorem, as we assumed that $\bA$ was non-trivial.

It remains to  bound the quantity $\ell$. With the notation of Theorem~\ref{T:FactorisationPetitePuissance}, if $\bE=\bA$, then $\cV=\Var \bA$. In the theorem, $\ell$ is given as
\begin{equation}\label{E:ell}
\ell=1+\max\setm{\ell_{\bS}}{\bS\in\cV\text{ and }\card{\bS}\text{ divides } N'}\,.
\end{equation}
where $N'=\card{F_{\cV}(N)}\times\card{A}$ for suitable $N$, and $\ell_{\bS}$ is a positive integer such that $\langle\cH(\bA^2,\bS);+\rangle$ has a generating family with $\ell_{\bS}$ elements, for each $\bS\in\cV$ with $\card\bS$ dividing $N'$. By Lemma \ref{L:SizefreeAlgModule},

$$\card{F_{\cV}(N)}\text{ divides }p_1^{N\alpha_1^2+ \alpha_1}\dots p_k^{N\alpha_k^2+\alpha_k}\,.$$
 Hence $N'$ divides $p_1^{N\alpha_1^2+2 \alpha_1}\dots p_k^{N\alpha_k^2+2 \alpha_k}$. Let $\bS$ such that $\card\bS$ divides $N'$, then
$
\card\bS\text{ divides }p_1^{N\alpha_1^2+2 \alpha_1}\dots p_k^{N\alpha_k^2+2 \alpha_k}\,.
$
Moreover $\card\bA=p_1^{\alpha_1}\dots p_k^{\alpha_k}$, therefore, it follows from Corollary~\ref{C:BoundH} that $\langle\cH(\bA^2,\bS);+\rangle$ has a generating family with
$$1+\max_{1\le i\le k} \left(N\alpha_i^3+ 2\alpha_i^2 \right)$$
elements. Hence we can pick
$$
\ell = 1+\max_{1\le i\le k} \left(N\alpha_i^3+ 2\alpha_i^2 \right)\,.
$$

In Lemma \ref{L:ExtensionsMorphismesAffine}, $N$ was given as  $1+\max\setm{N_y}{y\in A}$, where $N_y$ is as in (\ref{eq:N}). We already bounded the $N_y$ in the paragraph after (\ref{eq:N}) as
$N_y \le m \times (\sum_{i=1}^{k} r_i\alpha_i -1)$. Here, $r_i$ is the exponent of $p_i$ in the prime factorization of $\card \bR$ where $\bR$ is the ring of the module-variety $\cV_c$.
 Moreover, $m$ is the number of strict, non-trivial submodules of $\bR$ when considered as an $\bR$-module.

By Corollary \ref{C:ringsizeAbelian}, $r_i \le \alpha_i^2$.
With Lemma \ref{L:Majorations}, we see that $m \le p_1^{\alpha_1^4}\dots p_k^{\alpha_k^4}$. Therefore

$$N  \le 1+p_1^{\alpha_1^4}\dots p_k^{\alpha_k^4}\times \left(\sum_{i=1}^{k} \alpha_i^3 -1\right)\,.$$
Hence $\ell$ is smaller than the estimate in the statement of the Theorem. The result follows.
\end{proof}

\section{Example}\label{S:Examples}
We apply our results to an algebra whose examination was a crucial
in developing the proofs of the previous sections.

Consider the $8$-element ring $\bR=\cF_2[x,y]/I$, where $I$ is the ideal generated by $\set{x^2,y^2,xy,yx}$.
 Let $\bA$ be the module
that is obtained by considering $\bR$ as a module over itself.

By \cite{Gil14}, $\bA$ is dualizable by an alter ego that includes all compatible relation of size $28$.
By our main result, $\bA$ is strongly dualizable. A direct application of Theorem \ref{T:boundstrdual} will result in a very large bound of $702\cdot2^{81}+46$ on the arity of the partial operations in the alter ego.

By adopting the results of the previous sections to this specific example,
we can show that a lower bound suffices. As $\bA$ is a module, we may use the value $N$ (for $E= \bA$)  from (\ref{eq:N}), instead of $N_y$ from Lemma \ref{L:ExtensionsMorphismesAffine}. Moreover, instead of the bound on $N$ from the statement of Lemma \ref{L:ExtensionsMorphismesModulesI},
we can calculate $N$ directly by (\ref{eq:N}), and the $N_G$ according to the definition preceding it,
obtaining $N=14$. As $\bA$ generates a variety of modules, we see that $N'$ in Theorem \ref{T:FactorisationPetitePuissance} takes the value $8^{14}\cdot 8=2^{45}$. By Corollary~\ref{C:BoundH}, we may choose
  $\ell=1 + 3\cdot{45}$
in  Theorem \ref{T:FactorisationPetitePuissance}.
 Hence we may obtain a strong duality by using an alter ego with ``only'' the compatible partial operations of arity
$136$.

\section{Problems}\label{S:Problems}
We close with several problems motivated by our results.
\begin{problem} Which Abelian algebras that do not generate congruence-modular varieties are dualizable?
Which are fully and strongly dualizable?
\end{problem}
\begin{problem} Are nilpotent dualizable algebras (from congruence-modular varieties) always fully dualizable? Are they strongly dualizable?
\end{problem}
We remark that in many well-behaved classes of algebras, dualizabilty, full dualizability and strong dualizability coincide. Among nilpotent algebras, the results of \cite{BM13} show that in the subclass of supernilpotent algebras, all non-abelian algebras are non-dualizable (and by \cite{BM13} supernilpotence may be replaced by a slightly weaker condition).
\begin{problem}
Can the arity bound in our main theorem be improved upon?
\end{problem}
We conjecture that a bound of the form $(\log_2 |A|)^n$ suffices, for some fixed integer $n$.
\begin{problem}
Which Abelian algebras are strongly dualized by some alter ego that is a total structure?
\end{problem}

\section{Appendix: Counting homomorphisms and algebras} \label{S:appendix}

\begin{lemma}\label{L:Majorations}
Let $\bE,\bF$ be Abelian groups. Assume that $\card E=p_1^{\alpha_1}\dots p_k^{\alpha_k}$ and $\card F=p_1^{\beta_1}\dots p_k^{\beta_k}$, where the $p_i$ are distinct primes. Then the following statements hold.
\begin{enumerate}
\item The number of subgroups of $\bE$ is at most $p_1^{\alpha_1^2}\dots p_k^{\alpha_k^2}$.
\item $\card{\Hom(\bF,\bE)}$ divides $p_1^{\alpha_1\beta_1}\dots p_k^{\alpha_k\beta_k}$.
\end{enumerate}
\end{lemma}

\begin{proof}
Note that Abelian subgroups of $\bE$ are determined by their $p_i$-Sylow subgroups. Denote by $\bE_i$ the $p_i$-Sylow subgroup of $\bE$, such that $\card {E_i}=p_i^{\alpha_i}$. Each subgroup of $\bE_i$ has a family of generators with $\alpha_i$ elements. Therefore $\bE_i$ has at most $p_i^{\alpha_i^2}$ subgroups. It follows that $(1)$ holds.

We refer to \cite[Lemma~4.1]{Gil14} for $(2)$.
\end{proof}

\begin{lemma}\label{L:CountsAff}
Let $\bE,\bF$ be Abelian algebras. Assume that $\card E=p_1^{\alpha_1}\dots p_k^{\alpha_k}$ and $\card F=p_1^{\beta_1}\dots p_k^{\beta_k}$, where the $p_i$ are distinct primes. Then the following statements hold.
\begin{enumerate}
\item The number of subalgebras of $\bE$ is at most $p_1^{1+\alpha_1^2}\dots p_k^{1+\alpha_k^2}$.
\item $\card{\Hom(\bF,\bE)}$ divides $p_1^{(\alpha_1+1)\beta_1}\dots p_k^{(\alpha_k+1)\beta_k}$.
\end{enumerate}
\end{lemma}

\begin{proof}
For each $c\in E$, $x+_c y=t(x,c,y)$ induces an Abelian group structure on $E$. There are $\card E$ such structures. Let $\bA$ be a subalgebra of $\bE$, then for
$c \in A$, $\langle A; +_c \rangle$ is a subgroup of
$\langle E;+_c\rangle$.
 With Lemma \ref{L:Majorations}, the number of subalgebras of $\bE$ is at most $\card E\times  p_1^{\alpha_1^2}\dots p_k^{\alpha_k^2}=p_1^{1+\alpha_1^2}\dots p_k^{1+\alpha_k^2}$.

We refer to \cite[Lemma~4.1]{Gil14} for $(2)$.
\end{proof}

\begin{lemma}[\cite{Gil14}, Lemma 4.2]\label{L:generatorsAff}
Let $\bE$ be an Abelian algebra such that  $\card E=p_1^{\alpha_1}\dots p_k^{\alpha_k}$, where the $p_i$ are distinct primes, and let $N=1 + \max_{1\le i\le k}(\alpha_i)$. Then $\bE$ has a generating set with $N$ elements.
\end{lemma}

\begin{lemma}\label{L:countTermFunction}Let $\bE$ be an Abelian algebra, and $k$ a positive integer. Assume that $\card E=p_1^{\alpha_1}\dots p_k^{\alpha_k}$, where the $p_i$ are distinct primes. Then the number of distinct $k$-array term functions on $\bE$ divides $p_1^{k\alpha_1^2+\alpha_1}\dots p_k^{k\alpha_k^2+\alpha_k}$.
\end{lemma}

\begin{proof} Let $t$ be the ternary function witnessing the Abelianess of $\bE$. As $t$ is a homomorphism from $\bE^3$ to $\bE$, every $k$-array term function of $\bE$ is a homomorphism
from $\langle E^k;t\rangle$ to $\langle E;t\rangle$. The result now follows with Lemma \ref{L:CountsAff}, $(2)$.
\end{proof}

The following theorem is a particular case of Kearnes result in \cite{K91}, also see \cite[Corollary 4.4]{Gil14}.

\begin{theorem}\label{T:countSubIrred}
Let $\bA$ be finite Abelian algebra. Let $p_1^{\alpha_1}\dots p_k^{\alpha_k}$ be the prime decomposition of $\card A$. Let $\bS$ be a subdirectely irreducible algebra in $\Var\bA$. Then $\card S$ divides $p_1^{\alpha_1^2}\dots p_k^{\alpha_k^2}$.
\end{theorem}

\begin{lemma}\label{L:CardinalQuotient}
Let $\bA\subseteq\bB$ be Abelian algebras, then $\card{\bB/\Theta_A}=\card{\bB}/\card{\bA}$.
\end{lemma}
\begin{proof}
Abelian algebras are well known to have congruence classes of equal cardinality (see for example \cite{FM}, Corollary 7.5).  As $B$ is a congruence class of $\Theta_B$, the result follows.
\end{proof}

The following result may be easily shown using basic module theory.
\begin{lemma}\label{L:SizefreeAlgModule}
Let $\bR$ be a (unital) ring. Let $\cV$ be a variety of $\bR$-modules with constants, let $N$ be a positive integer.
Suppose that there exists $\bM \in \cV$, $a \in \bM$ such that  the action $r \mapsto ra$ from $\bR$ to $\bM$ is injective. Then
$\bF_{\cV}(N) = \bR^N\times \bF_{\cV}(0)$, where $\bR$ is considered as a module over itself.
\end{lemma}

\begin{corollary}
Let $\bA$ be an Abelian algebra, with $\card A= p_1^{\alpha_1}\dots p_k^{\alpha_k}$ where the
$p_i$ are disjoint primes. Set $\cV=\Var\bA$. Let $N\in\mathbb{N}$. Then
$$
\card{\bF_\cV(N)}\text{ divides }  p_1^{ N\alpha_1^2+\alpha_1}\dots p_k^{N\alpha_k^2+ \alpha_k}\,.
$$
\end{corollary}

\begin{proof}
Assume that $\bF$ is freely generated by $x_1, \dots,x_{N-1}, y$ in $\cV$, then by Lemma \ref{L:AjouterConstantes}, $\bF_y$ is freely generated by $x_1, \dots, x_{n-1}$ in $\cV_c$. As $\cV_c$ is term-equivalent to
a variety of modules with constants, by Lemma \ref{L:SizefreeAlgModule}, $\card{\bF_y}=\card {\bR^{N-1} \times F_{\cV_c}(0)}$, where $\bR$ is the ring constructed in lemma \ref{L:SeparerConstantes}. By Corollary
\ref{C:ringsizeAbelian}, $\card R$ divides $p_1^{ \alpha_1^2}\dots p_k^{\alpha_k^2}$.

It remains to estimate the cardinality of $\bF=\bF_{\cV_c}(0)$. If $y= c^{\bF_c(0)}$, then $\{y\}$ freely generates the reduct $\bF'$ of $\bF$ in $\cV$, once again by  Lemma \ref{L:AjouterConstantes}. As $\bA$ generates
$\cV$, $\bF'$ corresponds to all unary term functions on $\bA$. By Lemma \ref{L:countTermFunction}, $\card{\bF'}$ divides  $p_1^{ \alpha_1^2+\alpha_1}\dots p_k^{\alpha_k^2+ \alpha_k}$, and the result follows.
\end{proof}
\section*{Acknowledgements}
The first author  has received funding from the
European Union Seventh Framework Programme (FP7/2007-2013) under
grant agreement no.\ PCOFUND-GA-2009-246542 and from the Foundation for
Science and Technology of Portugal under  PCOFUND-GA-2009-246542, SFRH/BCC/52684/2014, and through the CAUL / CEMAT project.

\end{document}